\documentclass[runningheads,a4paper]{llncs}
\usepackage{etex}

\usepackage[T1]{fontenc}
\usepackage[utf8]{inputenc}
\usepackage{amsmath, amssymb}
\usepackage[textwidth=3.3cm]{todonotes}
\usepackage{float}
\usepackage{pgf, tikz}
\usepackage{hyperref}
\usetikzlibrary{arrows, patterns, snakes,calc}
\usepackage{graphicx}
\usepackage{url}
\usepackage{enumerate}
\usepackage{ifthen}
\usepackage{mathdots}
\usepackage[mathscr]{euscript}
\usepackage{mathrsfs}

\usepackage{ytableau} 
\usepackage[all]{xy} 

\usepackage{caption}
\usepackage{subcaption}
\usepackage{yfonts}
\usepackage{ bbold }

\newcommand{\N}{\ensuremath\mathbb{N}}

\newcommand{\T}{\ensuremath{\mathcal T}}

\newcommand{\lang}{\ensuremath\mathcal{L}}

\newtheorem{thm}{Theorem}

\renewenvironment{proof}{\paragraph{Proof.}}{\hfill$\square$\\}

\newcommand{\Ct}[1]{\tb{Cat}_{#1}}

\newcommand{\Ctp}[1]{\tb{Cat}_{(#1)}^{(+)}}
\newcommand{\Ctn}[1]{\tb{Cat}_{(#1)}^{(-)}}

\newcommand{\Hom}[1]{\H\left(#1\right)}

\def\rg{\textcolor{red}}

\def\a{\alpha}
\def\l{\lambda}

\def\De{\Delta}

\newcommand{\cmb}[2]{  \binom{ #1}{#2} }
\def\pgcd{\tb{gcd}}
\def\Bz{\mathcal{B}}
\newcommand{\Biz}[3]{\dfrac{1}{#1+#2}\cmb{#3#1+#3#2}{#3#1}}
\newcommand{\Bizu}[3]{\Bz_{#3_{1}}^{(#1,#2)}\Bz_{#3_{2}}^{(#1,#2)}\cdots\Bz_{#3_{l}}^{(#1,#2)}      }

\def\tb#1{\textbf{#1}}

\newcommand{\dint}[1]{\left\lfloor #1 \right\rfloor}

\newcommand{\dsum}[3]{ \displaystyle\sum\limits_{#1}^{#2} { #3}  }

\def\U{\mathbb{1}}

\def\D{\mathscr{D}}
\def\J{\ensuremath\mathcal{J}}

\def\T{\mathscr{T}}
\def\I{\mathscr{I}}
\def\H{\mathcal{H}}

\definecolor{lgray}{gray}{0.7}

\title{New formulas for Dyck paths in a rectangle}
\author{Jos\'{e} Eduardo Bla\v{z}ek$^1$
}

\urldef{\mailsa}\path|jeblazek@lacim.ca|

\institute{${}^1$Laboratoire de Combinatoire et d'Informatique Math\'{e}matique\\Universit\'{e} du Qu\'{e}bec \`{a} Montr\'{e}al\\
\mailsa \\
}

\authorrunning{Bla\v{z}ek}

\newcommand{\keywords}[1]{\par\addvspace\baselineskip
\noindent\keywordname\enspace\ignorespaces#1}
\begin{document}
\maketitle

\sloppy

\begin{abstract}
We consider the problem of counting the set of $\D_{a,b}$ of Dyck paths inscribed in a rectangle of size $a\times b$.  They are a natural generalization of the classical Dyck words enumerated by the Catalan numbers.

By using  Ferrers diagrams associated to Dyck paths, we derive formulas for  the enumeration of $\D_{a,b}$ with $a$ and $b$ non relatively prime, in terms of Catalan numbers.\end{abstract}
\section{Introduction}
The study of Dyck paths is a central topic in combinatorics as they provide one of the many interpretations of Catalan numbers. A partial overview can be found for instance in Stanley's comprehensive presentation of enumerative combinatorics \cite{Stanley} (see also \cite{BLL}).  As a language generated by an algebraic grammar is characterized in terms of a Dyck language, they are important in theoretical computer science as well \cite{Eilenberg}. On a two-letter alphabet they correspond to well parenthesized expressions and can be interpreted in terms of paths in a square. 
Among the many possible generalizations, it is natural to consider paths in a rectangle, see for instance  Labelle  and Yeh \cite{LY90}, and more recently Duchon \cite{Duchon} or Fukukawa \cite{Fu13}. In algebraic combinatorics Dyck paths are related to parking functions and the representation theory of the symmetric group \cite{GMV14}. The motivation for studying these objects stems  from this field in an attempt to better understand the links between these combinatorial objects. 

In this work, we obtain a new  formula for $|\D_{a,b}|$, when $a$ and $b$ are not relatively prime, in terms of the Catalan numbers using the notion of Christoffel path. More precisely, the main results of this article (\tb{diagrams decomposition method} in Section \ref{FDCM}, Theorems \ref{theo1} and \ref{theo2} in Section \ref{teo}) are formulas for the case where $a=2k$ :
 \[
 |\D_{a,b}| = \left \{ \begin{array}{ll}
 			\Ctn{a,n}-\dsum{j=1}{k-1}{\Ctn{a-j,n}\Ctn{j,n} }, \quad
										&\text{\rm if $b=a(n+1)-2$,}\\
			\Ctp{a,n}+\dsum{j=1}{k}{\Ctp{a-j,n}\Ctp{j,n} }, 
										&\text{\rm if $b=an+2$,}
			\end{array}
			\right.
\]
where $k,n \in \N$, $\Ctn{a,n}:=\Ct{(a,a(n+1)-1)},$ and $\Ctp{a,n}:=\Ct{(a,an+1)}$.

The paper is organized as follows. In Section \ref{sdn} we fix the notation for Dyck and Christoffel paths, and present their encoding by Ferrers diagrams. Then, in Section \ref{FDCM},  we develop the "Ferrers diagram comparison method" and "diagrams decomposition method". Section \ref{teo} contains several technical results in order to prove the main results, and in section \ref{exa} we present the examples.

\section{Definitions and notation}\label{sdn}

We borrow the notation from Lothaire \cite{Lothaire1}. An \emph{alphabet}  is a finite set $\Sigma$, whose elements are called \emph{letters}.  The set of finite words over
$\Sigma$ is denoted $\Sigma^*$  and $\Sigma^+ =
\Sigma^*\setminus \{\varepsilon\}$ is the set of nonempty words where   $\varepsilon \in \Sigma^*$ is  the empty word. 
The number of occurrences of a given letter $\alpha$ in the word $w$ is denoted $|w|_\alpha$  and
$|w|=\sum_{\alpha\in\Sigma}|w|_\alpha $ is the length of the word.
A \emph{language} is a subset $L\subseteq \Sigma^*$. The \emph{language} of a word $w$ is $\lang(w) = \{f \in \Sigma^* \mid w =
pfs,\ p,s\in \Sigma^*\}$,  and its elements are called the \emph{factors} of $w$.  

\paragraph{\bf Dyck words and paths.} It is well-known that the language of Dyck words on $\Sigma=\{\tb0,\tb1\}$ is the language generated by the algebraic grammar $D\rightarrow \tb0D\tb1D +\varepsilon$. They are enumerated by the Catalan numbers (see \cite{KOS09}),
\[\Ct{n}= \frac{1}{n+1}{2n \choose n},\]
and can be interpreted as lattice paths inscribed in a square of size  
$n\times n$ using down  and right unit steps (see Fig. \ref{E35D} (a)).

\vspace{-0.5cm}
 \begin{figure}[h] 
\centering
\subcaptionbox{}[0.4\linewidth]{
\begin{tikzpicture}[scale=0.4, every node/.style={scale=0.9}]
 \draw ( 0, 0 ) node{0}; 
\draw ( 0, -1 ) node{1};
\draw ( 0, -2 ) node{2};  
\draw ( 0,1 ) node{$_{(0,3)}$}; 
\draw ( 4,-3 ) node{$_{(3,0)}$};  
 \draw  ( 0.50,-2.5 ) rectangle ( 1.5, -1.5 );
\draw  ( 1.5, -2.5 ) rectangle ( 2.5, -1.5 ) ; 
\draw  [fill=gray!30 ]( 2.5, -2.5 ) rectangle (3.5, -1.5 ); 
\draw  ( 0.50, -1.5 ) rectangle ( 1.5, -0.50 ); 
\draw [fill=gray!30 ] ( 1.5, -1.5 ) rectangle (2.5, -0.50 ); 
\draw [fill= gray!30 ] (2.5, -1.5 ) rectangle ( 3.5,-0.50 );
\draw [fill=gray!30 ] ( 0.50, -0.50 ) rectangle (1.5, 0.50 ); 
\draw [fill= gray!30 ] (1.5, -0.50 ) rectangle ( 2.5,0.50 ); 
\draw [fill= gray!30 ] ( 2.5,-0.50 ) rectangle ( 3.5, 0.50 );
 \draw ( 1.5, 0.50 ) -- (1.5, -2.5 ); 
 \draw ( 2.5,0.50 ) -- ( 2.5, -2.5 ); 
 \draw (3.5, 0.50 ) -- ( 3.5,-2.5 ); 
 \draw [color= blue, dashed]+(0.5,0.5) -- ( 3.5, -2.5 ); \draw [line
width=1pt,color=red] ( 0.50, 0.50 ) -- (0.50, -0.50 ); \draw [line
width=1pt,color=red] ( 0.50, -0.50 ) -- (1.5, -0.50 ); \draw [line
width=1pt,color=red] ( 1.5, -0.50 ) -- (1.5, -1.5 ); \draw [line width=1pt,color=red]( 1.5, -1.5 ) -- ( 2.5,-1.5 ); 
\draw [line width=1pt,color=red] ( 2.5, -1.5 ) -- ( 2.5, -2.5 ); 
\draw[line width=1pt,color=red] ( 2.5, -2.5 ) -- (3.5, -2.5 ); 
\end{tikzpicture}
}
\subcaptionbox{}[0.4\linewidth]{
\begin{tikzpicture}[scale=0.4, every node/.style={scale=0.9}]
 \draw ( 0, 0 ) node{0}; 
\draw ( 0, -1 ) node{1};
\draw ( 0, -2 ) node{3}; 
\draw ( 0,1 ) node{$_{(0,3)}$}; 
\draw ( 6,-3 ) node{$_{(5,0)}$};  
\draw  ( 0.50,-2.5 ) rectangle ( 1.5, -1.5 );
\draw  ( 1.5, -2.5 ) rectangle ( 2.5, -1.5 ) ; 
\draw  ( 2.5, -2.5 ) rectangle (3.5, -1.5 ); 
\draw [fill= gray!30 ] (3.5, -2.5 ) rectangle ( 4.5,-1.5 ); 
\draw [fill= gray!30 ] ( 4.5,-2.5 ) rectangle ( 5.5, -1.5 );
\draw  ( 0.50, -1.5 ) rectangle ( 1.5, -0.50 ); 
\draw [fill=gray!30 ] ( 1.5, -1.5 ) rectangle (2.5, -0.50 ); 
\draw [fill= gray!30 ] (2.5, -1.5 ) rectangle ( 3.5,-0.50 );
 \draw [fill= gray!30 ] ( 3.5,-1.5 ) rectangle ( 4.5, -0.50 );
\draw [fill= gray!30 ] ( 4.5, -1.5 ) rectangle ( 5.5, -0.50 ); 
\draw [fill=gray!30 ] ( 0.50, -0.50 ) rectangle (1.5, 0.50 ); 
\draw [fill= gray!30 ] (1.5, -0.50 ) rectangle ( 2.5,0.50 ); 
\draw [fill= gray!30 ] ( 2.5,-0.50 ) rectangle ( 3.5, 0.50 );
\draw [fill= gray!30 ] ( 3.5, -0.50 )rectangle ( 4.5, 0.50 ); 
\draw [fill= gray!30] ( 4.5, -0.50 ) rectangle ( 5.5, 0.50 );
 \draw ( 0.50,-0.50 ) -- ( 5.5, -0.50 );
 \draw( 0.50, -1.5 ) -- ( 5.5,-1.5 ); 
 \draw ( 1.5, 0.50 ) -- (1.5, -2.5 ); 
 \draw ( 2.5,0.50 ) -- ( 2.5, -2.5 ); 
 \draw (3.5, 0.50 ) -- ( 3.5,-2.5 ); 
 \draw ( 4.5, 0.50 ) -- (4.5, -2.5 );
 \draw [color= blue, dashed]+(0.5,0.5) -- ( 5.5, -2.5 ); \draw [line
width=1pt,color=red] ( 0.50, 0.50 ) -- (0.50, -0.50 ); \draw [line
width=1pt,color=red] ( 0.50, -0.50 ) -- (
1.5, -0.50 ); \draw [line
width=1pt,color=red] ( 1.5, -0.50 ) -- (
1.5, -1.5 ); \draw [line width=1pt,color=red]
( 1.5, -1.5 ) -- ( 2.5,
-1.5 ); \draw [line width=1pt,color=red] ( 2.5
, -1.5 ) -- ( 3.5, -1.5 ); \draw
[line width=1pt,color=red] ( 3.5, -1.5 ) -- (
3.5, -2.5 ); \draw [line width=1pt,color=red]
( 3.5, -2.5 ) -- ( 4.5,
-2.5 ); \draw [line width=1pt,color=red] ( 4.5
, -2.5 ) -- ( 5.5, -2.5 );
\end{tikzpicture}
}
\caption{Dyck path and Ferrers diagram.}\label{E35D}
\end{figure}
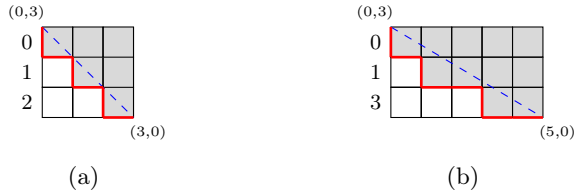
\vspace{-0.5cm}
More precisely an $(a,b)$-Dyck path is a south-east lattice path, going from $(0,a)$ to $(b,0)$, which stays below the $(a,b)$-diagonal, that is the line segment joining $(0,a)$ to $(b,0)$. In Figure \ref{E35D}, the paths are respectively $\tb0\tb1\tb0\tb1\tb0\tb1$  and $\tb0\tb1\tb0\tb1\tb1\tb0\tb1\tb1$. 

Alternatively such word may be encoded as a Ferrers diagram corresponding to the set of boxes to left (under) the path. As usual, Ferrers diagrams are identified by the number of boxes on each line, thus corresponding to partitions: 
 \begin{align}
 \l&=(\l_{a-1},\l_{a-2},\ldots,\l_{1}), &\text{  with  $ \l_{a-l}\leq\dint{\dfrac{bl}{a}}$ where  $1\leq l\leq a-1$}. \label{Ferrers}
\end{align}
\vspace{-0.4cm}

In the examples of Figure \ref{E35D}, the paths are respectively encoded by the sequences  $(2,1,0)$ and $(3,1,0)$.
The cases where $(a,b)$ are relatively prime, or $b=ak$ are of particular interest. For the case $b=ak$ with $k\geq 1$ we have the well-know formula of Fuss-Catalan (see \cite{KOS09}). 
\[
\Ct{(a,k)}=\dfrac{1}{ak+1}\cmb{ak+a}{a}.
\]
For $a\times b$ rectangles, with $a$ and $b$ are relatively prime, we also have the "classical" formula:
\[
\Ct{(a,b)}=\dfrac{1}{a+b}\cmb{a+b}{a}.
\]

In particular, when $a=p$ is  prime, either $b$ and $p$ are relatively prime, or $b$ is a multiple of $p$. Hence the relevant number of Dyck paths is:  
\begin{align*}
|\D_{p,b}|=\begin{cases}
\frac{1}{p+b}\cmb{p+b}{p} &\mbox{ if } \tb{gcd}(p,b)=1, \\
\frac{1}{p+b+1}\cmb{p+b+1}{p} &\mbox{ if } b=kp.
\end{cases}
\end{align*} 
The generalized ballot problem is related with the number of lattice paths form $(0,0)$ to $(a,b)$ that never go below the line $y=kx$ (see \cite{Serrano03}):
\begin{align*}
\dfrac{b-ka+1}{b}\cmb{a+b}{a}&&\text{where $k\geq 1$,  and $b>ak\geq0$.}
\end{align*}
And the number of lattice paths of length $2(k+1)n+1$ that start at $(0,0)$ and that avoid touching or crossing the line $y=kx$ (see \cite{Chap09}) has the formula:
\begin{align*} 
\cmb{2(k+1)n}{2n}-(k-1)\dsum{2n-1}{i=0}{\cmb{2(k+1)n}{i}},&&\text{where $n\geq 1$ and $k\geq0$.}
\end{align*} 
In the more general case we have a formula due to Bizley (see \cite{Biz54}) expressed as follows.
Let $m=da$, $n=db$ and $d=\tb{gcd}(m,n)$, then:
\begin{align*}
\Bz_{k}^{(a,b)}& :=\Biz{a}{b}{k} &\text{for $k\in \N$}, \\
\Bz_{\l}^{a,b}& :=\Bizu{a}{b}{\l} &\text{if $\l=(\l_1,\l_2,\dots,\l_l)$,}
\end{align*}
It is straightforward to show that the number of Dyck paths in $m\times n$ is:
\begin{align*}
|\D_{m,n}|&:= \dsum{\l \vdash d}{}{\frac{1}{z_{\l}}\Bz_{\l}^{(a,b)}}&\text{where $n\geq 1$ and $k\geq0$.}
\end{align*}
\vspace{-0.5cm}


\paragraph{\bf Christoffel paths and words.} A Christoffel path between two distinct points $P=(0,k)$  and $P'=(0,l)$ on a rectangular grid $a\times b$ is the closest lattice path that stays strictly below the segment $PP'$ (see \cite{MR13}).  For instance, the Dyck path of Figure \ref{E35D}(b) is also Christoffel, and the associated word is called a Christoffel word. The Christoffel path of a rectangular grid $a\times b$ is the Christoffel path associated to the line segment going from the north-west corner to the south-east corner of the rectangle of size $a\times b$.
 As in the case of Dyck paths, every Christoffel path in a fixed rectangular grid $a\times b$ is identified by a Ferrers diagram of shape $(\l_{a-1},\l_{a-2},\ldots,\l_{1})$ given by Equation \eqref{Ferrers}.
 

For later use, we define two functions associated to Ferrers diagram. 
Let $Q_{a, b}$ to be the total number of boxes in the Ferrers diagram associated to the Christoffel path of $ a\times b $ (see \cite{BJE14}):
\begin{align}
Q_{a,b} &=\dfrac{(a-1)(b-1)+\tb{gcd}(a,b)-1}{2}.\label{Qab}
\end{align} 

Also, let $ \De_ {a, b} (l) $ be the difference between the boxes of  the Ferrers diagrams associated to the Christoffel paths of $ a\times b $ and $ a \times (b-1) $, respectively:

\begin{align*}
\De_{a,b}(l):=\dint{\dfrac{bl}{a}}-\dint{\dfrac{(b-1)l}{a}},
\end{align*}
where $a<b\in\N$ and $1\leq l \leq a-1$.

In the next section we give an alternate method to calculate the number of ($a,b$)-Dyck paths when $a$ and $b$ are not relatively prime, and satisfying certain conditions in terms of the Catalan numbers.
\paragraph{\bf{Isosceles diagrams.}} An isosceles diagram $\I_{n}$ is a Ferrers diagram associated to a Christoffel path in a square having side length $n$. Given a Ferrers diagram $\T_{a,b}$, we call \textit{maximum isosceles diagram} the largest isosceles diagram included in $\T_{a,b}$.

\paragraph{\bf{Ferrers set.}} Let $\T_{a,b}$ be a Ferrers diagram. The Ferrers set of $\T_{a,b}$ is the set of all Dyck paths contained in $\T_{a,b}$.
\section{Ferrers diagrams comparison method}\label{FDCM}

Let $\T_{a,b}$ be the Ferrers diagram associated to a Christoffel path of $a\times b$. In order to establish  the main results we need to count the boxes in excess between the  Christoffel paths in rectangles  $a\times b$ and $a\times c$, for any $c>b$. We develop a method to do this by removing exceeding boxes between  $\T_{a,b}$ and $\T_{a,c}$, for $c>b$.
Using the functions $Q_{a,b}$ and $\De_{a,b}(l)$, our comparison method gives the following rules:
\begin{enumerate}[\tb{Rule} 1:]
\item\label{R1} If $Q_{a,b}=1$ and $\De_{a,b}(i)=1$, there is only one corner in $\T_{a,c}$ which does not belong to the $\T_{a,b}$. Let $\T_{a_1,b_1}$ and $\T_{a'_1,b'_1}$ be the Ferrers diagram obtained by erasing from $\T_{a,c}$ the row and the column that contain $\rg{ \a}$ (see Figure \ref{DifDiag}(b)).

\begin{figure}[h]
\centering
\subcaptionbox{Comparison $\T_{a,b}$ and $\T_{a,c}$ }[0.5\linewidth]{
$
\ytableausetup{mathmode, boxsize=0.8em }
\begin{ytableau}
\none [^{_{\l_{n}}}\,\,\,\, ]&  \, &\,&\\
\none [^{_{\l_{n-1}}}\,\,\,\, ]&  \,& \, && \, &\\
\none [^{_{\l_{n-2}}} \,\,\,\,]&  \,& \, \,& \, & \,  & & \, & \\
\none [^{_{\vdots}} \,\,\,\,]& \none[^{_{\vdots}}]& \none[^{_{\vdots}}] \,& \none[^{_{\vdots}}] &
\none[^{_{\vdots}}]  &\none[^{_{\vdots}}] & \none[^{_{\vdots}}]& \none[^{_{\vdots}}] \\
\none [^{_{\l_{i}}}\,\,\,\, ]&  \,& \, \,& \, & \,  & \,& \, & & \, & &*(red)&\none[ \rg{ \a}] \\
\none [^{_{\vdots}} \,\,\,\,]& \none[^{_{\vdots}}]& \none[^{_{\vdots}}] \,& \none[^{_{\vdots}}] &
\none[^{_{\vdots}}]  &\none[^{_{\vdots}}] & \none[^{_{\vdots}}]& \none[^{_{\vdots}}] &\none[^{_{\vdots}}] & \none[^{_{\vdots}}]& \none[^{_{\vdots}}] \\
\none [^{_{\l_{2}}} \,\,\,\,]&   \,& \, \,& \, & \,  & \,& \, & \,&\, &\,& & \,&  \\
\none [^{_{\l_{1}}} \,\,\,\,]&   \,& \, \,& \, & \,  & \,& \, & \,&\, &\,& &\, &\, & & \\
\end{ytableau}
$} 
\subcaptionbox{$\T_{a_1,b_1}$ and $\T_{a'_1,b'_1}$}[0.3\linewidth]{
$
\ytableausetup{mathmode, boxsize=0.8em }
\begin{ytableau}
\none [^{_{\l_{n}}}\,\,\,\, ]&  \, &\,&\\
\none [^{_{\l_{n-1}}}\,\,\,\, ]&  \,& \, && \, &\\
\none [^{_{\l_{n-2}}} \,\,\,\,]&  \,& \, \,& \, & \,  & & \, & \\
\none [^{_{\vdots}} \,\,\,\,]& & &  &   &  & & &  & \\
\none [^{_{\l_{i}}}\,\,\,\, ]& \none[]& \none[] \,& \none[] & \none[]  & \none[]& \none[] & \none[]& \none[] &\none[] &*(red)  &\none[ \rg{ \a}] \\
\none [^{_{\vdots}} \,\,\,\,]& \none[]& \none[] \,& \none[] &\none[]  &\none[] & \none[]& \none[] &\none[] & \none[]& \none[^{_{\vdots}}] \\
\none [^{_{\l_{2}}} \,\,\,\,]&   \none[]& \none[] \,& \none[] & \none[]  & \none[]& \none[] & \none[]&\none[] &\none[]&\none[] & &  \\
\none [^{_{\l_{1}}} \,\,\,\,]&  \none[]& \none[] \,& \none[] & \none[]  & \none[]& \none[] & \none[]&\none[] &\none[]&\none[] && & & \\
\end{ytableau}
$
}
\vspace{-0.4cm}
\caption{Rule 1.}\label{DifDiag}  
\end{figure}

These Ferrers diagrams are not associated to a Christoffel path in general. Let \[\J_{a_1,b_1}\subseteq \D_{a_1,b_1}, \text{  and  } \J_{a'_1,b'_1}\subseteq \D_{a'_1,b'_1}\]  be the sets of Dyck paths contained in the Ferrers diagrams $\T_{a_1,b_1}$ and $\T_{a'_1,b'_1}$, respectively. We have:
\[|\D_{a,c}|-|\D_{a,b}|=-|\J_{a_1,b_1}|\cdot|\J_{a_1',b_1'}|,\]

It is clear that if the box $\a$ is located on the bottom line ($l=a-1$), the  equation is reduced to:
\[|\D_{a,c}|-|\D_{a,b}|=-|\J_{a_1,b_1}|.\]
\item\label{R2} 
When $Q_{a,b}=k$ and there are exactly $k$ rows with a difference of one box we need to calculate how many paths contain these boxes (see Figure \ref{DifDiag2}), so we construct a sequence of disjoint sets as follows. Let $A_j $ be the set of all paths that do not contain the boxes $\rg{\a_i}$ for each $ i> j $, where $1\leq j\leq k$.  Also, let $B_j$ be the set of all paths that do not contain the boxes $ \a_i$ for each $i <j $, where $ 1 \leq j \leq k $.

\begin{figure}[h]
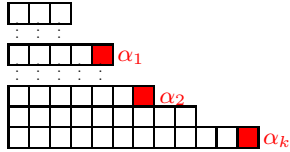

$$
\ytableausetup{mathmode, boxsize=0.8em}
\begin{ytableau}
\none [  ]&  \, &&\\
\none [ ]&\none [^{_{\vdots}}] &\none [^{_{\vdots}}]&\none [^{_{\vdots}}]\\
\none []&  \,& && \ &*(red) & \none[\ \ \rg{\a_1}]\\
\none [ ]&\none [^{_{\vdots}}] &\none [^{_{\vdots}}]&\none [^{_{\vdots}}] &\none [^{_{\vdots}}]&\none [^{_{\vdots}}]\\
\none []&  \,&  \,&  &   & &  & *(red) & \none[\ \ \rg{\a_2}] \\
\none [ ]&  \,& \,&  &   & &  & & &  \\
\none []&   \,&   &   & && & && & & &*(red)& \none[\ \ \rg{\a_k}] \\
\end{ytableau}
$$ 
\vspace{-0.4cm}
\caption{More one box.}\label{DifDiag2}
\end{figure}

\vspace{-0.5cm}
This strategy gives us disjoint sets that preserve the total union, so using \tb{Rule \ref{R1}} for every $A_j$ or $B_j$ we get:
\[|\D_{a,c}|-|\D_{a,b}|=-\dsum{j=1}{k}{(|\J_{a_j,b_j}|\cdot|\J_{a_j',b_j'}|)},\]

where $\J_{a_j,b_j}\subseteq \D_{a_j,b_j}$, and $\J_{a'_j,b'_j}\subseteq \D_{a'_j,b'_j}$.
\end{enumerate}


\subsection{Diagrams decomposition method}

Using the diagrams comparison method we make an iterative process erasing boxes in  excess between the diagram $\T_{a,b}$ and its respective maximum isosceles diagram $\I_{n}$. It begins at the right upper box as shown in Figure \ref{pas1}. The decomposition is give in sums and products of diagrams.
The sum operation $+$ is given by the union of disjoint Ferrers sets. We can consider a red box in the border of the diagram. Any path contained in a diagrams having the red box is written as one of the two cases in Figure \ref{suma}.
\vspace{-0.4cm}
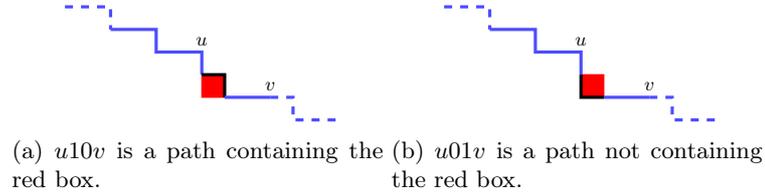
\begin{figure}[h]
\centering
\subcaptionbox{$u10v$ is a path containing the red box.}[0.4\linewidth]{
\begin{tikzpicture}[scale=.3, every node/.style={scale=0.8}] 
\draw[fill=red,draw=red] ( 2,-1 ) -- ( 2,-2 ) -- (3,-2) -- (3,-1)-- cycle; 
\draw[color=blue!70,dashed,line width=1.2pt](-4,2)--(-2,2)--(-2,1);
\draw[color=blue!70,dashed,line width=1.2pt](5,-2)--(6,-2)--(6,-3)--(8,-3);
\draw[color=blue!70,line width=1.2pt](-2,1)--(0,1)--(0,0) -- ( 2, 0 ) -- ( 2,-1 );
\draw[color=black,line width=1.2pt] ( 2,-1 )--(3,-1)--(3,-2);
\draw[color=blue!70,line width=1.2pt](3,-2)--(5,-2); 
\draw [ black ]  ( 2,0 )  node[above] {$u$ };
\draw [ black ]  ( 5,-2 )  node[above] {$v$ };
\end{tikzpicture}
}
\subcaptionbox{$u01v$ is a path not containing the red box.}[0.4\linewidth]{
\begin{tikzpicture}[scale=.3, every node/.style={scale=0.8}] 
\draw[fill=red,draw=red] ( 2,-1 ) -- ( 2,-2 ) -- (3,-2) -- (3,-1)-- cycle; 
\draw[color=blue!70,dashed,line width=1.2pt](-4,2)--(-2,2)--(-2,1);
\draw[color=blue!70,dashed,line width=1.2pt](5,-2)--(6,-2)--(6,-3)--(8,-3);
\draw[color=blue!70,line width=1.2pt](-2,1)--(0,1)--( 0, 0 ) -- ( 2, 0 ) -- ( 2,-1 );
\draw[color=black,line width=1.2pt]( 2,-1 )--(2,-2)--(3,-2); 
\draw[color=blue!70,line width=1.2pt](3,-2)--(5,-2); 
\draw [ black ]  ( 2,0 )  node[above] {$u$ };
\draw [ black ]  ( 5,-2 )  node[above] {$v$ };
\end{tikzpicture}
}
\caption{Separation of diagrams.}\label{suma}
\end{figure}
\vspace{-0.6cm}
The products of diagrams $\T\times\T'$  is a diagram containing all possible concatenation of a Dyck path of $\T$ with a Dyck path of $\T'$. For example, the diagram corresponding to $\D_{4,6}$ is $[4,3,1]$ $\ytableausetup {mathmode, boxsize= 0.5 em,centertableaux} \ydiagram [ ]{ 1 , 3 , 4 }$ includes the isosceles diagram $[3,2,1]$ $\ytableausetup {mathmode, boxsize= 0.5 em,centertableaux} \ydiagram [*(yellow!70) ]{ 1 , 2 , 3 }$. When we remove the box the diagram splits into two pairs associated with operations that simplify the computation of paths (see Figure \ref{pas1}).

\vspace{-1,2cm}
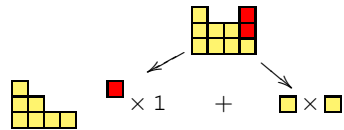
\begin{figure}[h] 
\begin{align*}
\xymatrixrowsep{0.05in}
\xymatrixcolsep{0.05in}
\xymatrix{
&\ytableausetup {mathmode, boxsize=0.6em,centertableaux} 
\ydiagram[*(red)]
  {1+0,2+1,3+1}
*[*(yellow!70)]{1,3,4}
\ar@{->}[ld]^{} \ar@{->}[rd]^{}  \\
\ytableausetup {mathmode, boxsize= 0.6em,centertableaux} 
\ydiagram[*(red)]
  {1+0,2+0,3+1}
*[*(yellow!70)]{1,2,4}
\times \U& +  &\ytableausetup {mathmode, boxsize= 0.6em,centertableaux} \ydiagram [ *(yellow!70)]{ 1 }  \times \ydiagram [ *(yellow!70)]{ 1 }& 
}
\end{align*}
\vspace{-0.4cm}
\caption{First diagram decomposition.}\label{pas1}
\end{figure}
\vspace{-0.7cm}
In Figure \ref{pas1}, $\U$ is an empty diagram. We repeat this method until all the diagrams are isosceles (the operation $\times$ distributes the operation $+$).
\vspace{-0.8cm}
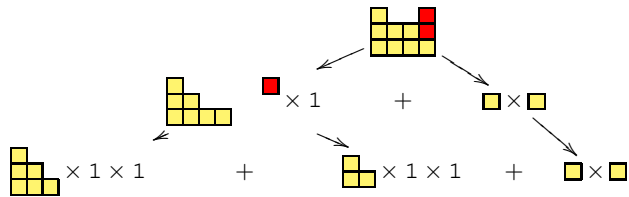
\begin{figure}[h] 
\begin{align*}
\xymatrixrowsep{0.03in}
\xymatrixcolsep{0.03in}
\xymatrix{
&&\ytableausetup {mathmode, boxsize= 0.6em,centertableaux}
\ydiagram[*(red)]
  {1+0,2+1,3+1}
*[*(yellow!70)]{1,3,4}
\ar@{->}[ld]^{} \ar@{->}[rd]^{}  \\
&\ytableausetup {mathmode, boxsize= 0.6em,centertableaux}
\ydiagram[*(red)]
  {1+0,2+0,3+1}
*[*(yellow!70)]{1,2,4}\times \U\ar@{->}[ld]^{}\ar@{->}[rd]^{}& +  &\ytableausetup {mathmode, boxsize= 0.6em ,centertableaux} \ydiagram [ *(yellow!70)]{ 1 }  \times \ydiagram [ *(yellow!70)]{ 1 } \ar@{->}[rd]^{}& \\
\ytableausetup {mathmode, boxsize= 0.6em,centertableaux} \ydiagram [ *(yellow!70)]{ 1 , 2 , 3 } \times \U\times \U&+& \ydiagram [ *(yellow!70)]{ 1 , 2  }\times \U\times \U &+&\ydiagram [ *(yellow!70)]{ 1 }  \times \ydiagram [ *(yellow!70)]{ 1 }&&\\
}
\end{align*}
\vspace{-0.4cm}
\caption{Full diagram decomposition.}\label{dec2}
\end{figure}

\vspace{-0.7cm}
Clearly, we can count the Dyck paths in an isosceles diagram with a classical Catalan formula because there is a relation between the decomposition and the number of Dyck paths. This relation, denoted $\H$, between the isosceles diagrams and Catalan numbers is such as:
$$
\begin{array}{c c}
\H(\U):=1,& \H(\I_n+\I_m):=\H(\I_n)+\H(\I_m),\\
\H(\I_n):=\Ct{n},& \H(\I_n\times\I_m):=\H(\I_n)\cdot\H(\I_m).
\end{array}
$$

\begin{align*}
\Hom{\ytableausetup {mathmode, boxsize= 0.5 em,centertableaux} \ydiagram [ *(yellow!70)]{ 1 , 3 , 4 } }
& = \Hom{\ytableausetup {mathmode, boxsize= 0.5 em,centertableaux} \ydiagram [ *(yellow!70)]{ 1 , 2 , 3 } \times 1\times 1+\ytableausetup {mathmode, boxsize= 0.5 em,centertableaux} \ydiagram [ *(yellow!70)]{ 1 , 2  }\times 1\times 1 +\ytableausetup {mathmode, boxsize=0.5 em,centertableaux} \ydiagram [ *(yellow!70)]{ 1 }  \times \ydiagram [ *(yellow!70)]{ 1 }} &\\
& = \Hom{\ytableausetup {mathmode, boxsize= 0.5 em,centertableaux} \ydiagram [ *(yellow!70)]{ 1 , 2 , 3 } \times 1\times 1}+\Hom{\ytableausetup {mathmode, boxsize= 0.5 em,centertableaux} \ydiagram [ *(yellow!70)]{ 1 , 2  }\times 1\times 1} +\Hom{\ytableausetup {mathmode, boxsize=0.5 em,centertableaux} \ydiagram [ *(yellow!70)]{ 1 }  \times \ydiagram [ *(yellow!70)]{ 1 }} &\\
& = \Hom{\ytableausetup {mathmode, boxsize= 0.5 em,centertableaux} \ydiagram [ *(yellow!70)]{ 1 , 2 , 3 } }+\Hom{\ytableausetup {mathmode, boxsize= 0.5 em,centertableaux} \ydiagram [ *(yellow!70)]{ 1 , 2  }} +\Hom{\ytableausetup {mathmode, boxsize=0.5 em,centertableaux} \ydiagram [ *(yellow!70)]{ 1 } } \times \Hom{\ydiagram [ *(yellow!70)]{ 1 }} &\\
&= \Ct{4}+\Ct{3}+\Ct{2}\times\Ct{2} =23&\\
\end{align*}

\subsection{Technical results}
The following technical formulas are needed in the sequel (see \cite{BJE14}).
\begin{enumerate}[$i)$]
\item Let $a=2k$, $b=2k(n+1)-1$, and $1\leq l \leq 2k-1$. Then,
\begin{align}
\De_{2k,2k(n+1)-1}(l)=\begin{cases}
0 &\mbox{if } 1\leq l\leq k,\\
1&\mbox{if } k+1\leq l\leq 2k-1.
\end{cases}\label{eDeltaS}
\end{align} 
\item Let $a=2k$, $b=2kn+2$, and $1\leq l \leq 2k-1$. Then,
\begin{align}
\De_{2k,2kn+2}(l)=\begin{cases}
0 &\mbox{if } 1\leq l\leq k-1,\\
1&\mbox{if } k\leq l\leq 2k-1.
\end{cases}\label{eDeltaI}
\end{align}

\item Let $a,k\in \N$, and $k<a$ . There exists a unique $r\in\N$ such as for $k=1,\ldots, a-1$ :
\begin{align}
\dint{\dfrac{kr}{a}}=k-1 & \text{  or  } \dint{\dfrac{kr}{a}}=k, \label{eSPE}
\end{align}
and $\pgcd(r,a)=1$. The solution is given by $r=a-1$
\end{enumerate}

\section{Theorems}\label{teo}
Now we are ready to prove  the two main results of this article. The formulas are obtained by studying  the "Ferrers diagram comparison method" (see section 3), in the cases where every Ferrers diagram obtained by subdivisions of $\T_{a,c }$ is associated  to a Christoffel path inscribed in a rectangular box of co-prime dimension.

\begin{thm}[see \cite{BJE14}]\label{theo1}
Let $a=2k$, $b=a(n+1)-2$, and $k,n \in \N$, then the number of Dyck paths is:
\[
 |\D_{a,b}| = \Ct{(a,n)}^{(-)}-\dsum{j=1}{k-1}{\Ct{(a-j,n)}^{(-)}\Ct{j,n}^{(-)} },
\]
where $\Ct{a,n}^{(-)}:=\Ct{(a,a(n+1)-1)}.$
\end{thm}

\begin{proof}
From Equations \ref{Qab} and \ref{eDeltaS}, we get that there are $k-1$ total difference between the Ferrers diagram associated to the Christoffel path of $a\times b$ and $a\times (b+1)$. We easily obtain that the Ferrers diagram associated to $\D_{c,cn+c-1}$ is $\l=((c-1)n+c-2,\dots,2n+1,n)$.

By  definition of $A_j$, the rectangles $j \times b_1$ and $(2k-j)\times b'_1$ are such that their maximal underlying diagrams are:
\begin{align*}
 \l&=((j-1)n+j-2,\dots,3n+2,2n+1,n),&\\
 \l'&=((2k-j-1)n+2k-j-2,\dots,2n+1,n),
 \end{align*}
respectively. So,
\[|A_j|=|\D_{2k-j,(2k-j)n+2k-j-1}||\D_{j,jn+j-1}|.\]
Since all rectangles are relatively prime, we have:
\[|\D_{a,b}|-|\D_{a,b+1}|=-\dsum{j=1}{k-1}{(|\D_{a-j,(a-j)(n+1)-1}|\cdot|\D_{j,j(n+1)-1}|)}.\]
then
\[
 |\D_{a,b}| = \Ct{(a,n)}^{(-)}-\dsum{j=1}{k-1}{\Ct{(a-j,n)}^{(-)}\Ct{j,n}^{(-)} }.
\]
where $\Ct{t,n}^{(-)}:=\Ct{(t,t(n+1)-1)}$, of course $\Ct{(1,n)}^{(-)}=1$.
\end{proof}

\begin{thm}[see \cite{BJE14}]\label{theo2}
Let $a=2k$, $b=an+2$ and $k,n \in \N$, then the number of Dyck paths is:
\[
 |\D_{a,b}| = \Ct{(a,n)}^{(+)}+\dsum{j=1}{k}{\Ct{(a-j,n)}^{(+)}\Ct{j,n}^{(+)} }.
\]
where $\Ct{a,n}^{(+)}:=\Ct{a,an+1}$.
\end{thm}

\begin{proof}From Equations \ref{Qab} and \ref{eDeltaI}, we get that there are $k$ total difference between the Ferrers diagram associated to the Christoffel path of $a\times b$ and $a\times (b+1)$. We easily get that the Ferrers diagram associated to $\D_{c,cn+1}$ is $\l=((c-1)n,\dots,2n,n)$. By definition of $B_j$, the rectangles $j \times b_1$ et $(2k-j)\times b'_1$ are such as  that their maximal underlying diagram are:
\begin{align*}
 \l&=((j-1)n,\dots,3n,2n,n),&\\
 \l'&=((2k-j-1)n,\dots,2n,n),
 \end{align*}
 respectively. So,
\[|B_j|=|\D_{2k-j,(2k-j)n+1}||\D_{j,jn+1}|.\]
Since all rectangles are relatively prime, we have:
\[|\D_{a,b}|-|\D_{a,b-1}|=-\dsum{j=1}{k-1}{(|\D_{a-j,(a-j)n-1}|\cdot|\D_{j,jn+1}|)}.\]
then
\[
 |\D_{a,b}| = \Ct{(a,n)}^{(+)}+\dsum{j=1}{k}{\Ct{(a-j,n)}^{(+)}\Ct{j,n}^{(+)} }.
\]
where $\Ct{t,n}^{(+)}:=\Ct{(t,tn+1)}$, of course $\Ct{(1,n)}^{(+)}=1$.
\end{proof}

\section{Examples}\label{exa}
In order to illustrate the main results, we consider the cases $ \D_{8,8n + 6}$, to generalize the case of discrepancies with a larger diagram. Then we  study $\D_{6,6n + 2}$ to generalize the case of discrepancies for shorter diagram. These cases corresponding to rectangles having relatively prime dimensions such that $ \J_{a_j, b_j} = \D_{a_j, b_j} $ and $\J_{a'_j, b'_j} = \D_{a'_j, b'_j} $. Finally, we give some examples of the diagrams decomposition method.
\vspace{-0.4cm}
\subsection{Example $\D_{8,8n+6}$} \label{8n+6}
We apply the comparison method to $\D_{8,8n + 6}$ and $\D_{8,8n + 7} $. Using Equation \ref{Qab}, we get that the difference in total number of sub-diagonal boxes is $Q_{8,8n + 7} -Q_{8,8n + 6} =  3$. To find the lines where they are located we use the Equation \ref{eDeltaS}. In this cases $\De(l)$ is zero except for $l=5,6,7$ (see Figure \ref{D867}).
\vspace{-0.6cm}
\begin{figure}[h]
$$
\ytableausetup{mathmode, boxsize=1.em}
\begin{ytableau}
\none [^{_{n}}\,\,\,\, ]&  \, &\none[_{\cdots}]&\\
\none [^{_{2n+1}}\,\,\,\, ]&  \,& \none[_{\cdots}] && \none[_{\cdots}] &\\
\none [^{_{3n+2}} \,\,\,\,]&  \,& \none[_{\cdots}] \,& \none[_{\cdots}] & \none[_{\cdots}]  & & \none[_{\cdots}] & \\
\none [^{_{4n+3}}\,\,\,\, ]&  \,& \none[_{\cdots}] \,& \none[_{\cdots}] & \none[_{\cdots}]  & \none[_{\cdots}]& \none[_{\cdots}] & & \none[_{\cdots}] &  \\
\none [^{_{5n+3}} \,\,\,\,]&   \,&  \none[_{\cdots}] & \none[_{\cdots}]  & \none[_{\cdots}]& \none[_{\cdots}] & \none[_{\cdots}]&\none[_{\cdots}] &\none[_{\cdots}]& &\none[_{\cdots}] & &*(red) &\none[\a_1] \\
\none [^{_{6n+4}} \,\,\,\,]&    \,& \none[_{\cdots}] & \none[_{\cdots}]  & \none[_{\cdots}]& \none[_{\cdots}] & \none[_{\cdots}]&\none[_{\cdots}] &\none[_{\cdots}]&\none[_{\cdots}] &\none[_{\cdots}]&\none[_{\cdots}]& &\none[_{\cdots}] & &*(red) &  \none[\a_2] \\
\none [^{_{7n+5}} \,\,\,\,]&    \,& \none[_{\cdots}] & \none[_{\cdots}]  & \none[_{\cdots}]& \none[_{\cdots}] & \none[_{\cdots}]&\none[_{\cdots}] &\none[_{\cdots}]&\none[_{\cdots}] &\none[_{\cdots}] &\none[_{\cdots}] &\none[_{\cdots}] &\none[_{\cdots}]&\none[_{\cdots}]& &\none[_{\cdots}] & &*(red) &\none[\a_3]\\ 
\end{ytableau}
$$ 
\vspace{-0.4cm}
\caption{$\D_{8,8n+6}$ and $\D_{8,8n+7}$ }\label{D867}
\end{figure}

Applying Rule \ref{R2}, we have that $A_1=\{ $ paths containing the box $\a_1\text{ and not }(\a_2 \text{ or } \a_3)\}$, $A_2 =\{ $ paths containing the box $\a_2 \text{ and not }\a_3\}$, and
$A_3 =\{ \text{paths containing the box } \a_3 \}$. 

Applying the rule \ref{R1} to these sets, we have:
\begin{enumerate}[\tb{Case} 1:]
\item
for $A_1$, we must find the rectangles $5 \times b_1$ and $3\times b'_1$
with underlying diagram $\l=(5n+4,4n+3,3n+2,2n+1,n)$ and $\l'=(n)$, then
\[|A_1|=|\D_{5,5n+4}||\D_{3,3n+2}|.\]

\item
for $A_2$, we must find the rectangles $6 \times b_2$ and $2\times b'_2$ with underlying diagram $\l=(5n+4,4n+3,3n+2,2n+1,n)$ and $\l'=(n)$, then
\[|A_2|=|\D_{6,6n+5}||\D_{2,2n+1}|.\]

\item
for $A_3$, we must find the rectangle $7 \times b_3$ with underlying diagram $\l=(6n+5,5n+4,4n+3,3n+2,2n+1,n)$, then,
\[|A_3|=|\D_{7,7n+6}|.\]
\end{enumerate}
Finally, we obtain:
\begin{align*}
|\D_{8,8n+6}| = \Ct{(8,8n+7)}&-\Ct{(7,7n+6)}\\
& -\Ct{(6,6n+5)}\Ct{(2,2n+1)} - \Ct{(5,5n+4)}\Ct{(3,3n+2)}.
\end{align*}

\subsection{Example $\D_{6,6n+2}$} \label{6n+2}

Similarly to the previous example, comparing $\D_{6,6n + 2} $ and $\D_{6,6n + 1} $ from Equation \ref{Qab}, we get that the total difference is $Q_{6,6n + 2}-Q_{6,6n + 1} =  3$. From Equation \ref{eDeltaI}, in this cases $\De(l)$ is zero except for $l = 3,4,5$ (see Figure \ref{D612}).

\begin{figure}[h]
$$
\ytableausetup{mathmode, boxsize=1.2em}
\begin{ytableau}
\none [^{_{n}}\,\,\,\, ]&  \, &\none[_{\cdots}]&\\
\none [^{_{2n}}\,\,\,\, ]&  \,& \none[_{\cdots}] && \none[_{\cdots}] &\\
\none [^{_{3n}} \,\,\,\,]&  \,& \none[_{\cdots}] \,& \none[_{\cdots}] & \none[_{\cdots}]  & & \none[_{\cdots}] & &*(red)&\none[\a_1]\\
\none [^{_{4n}}\,\,\,\, ]&  \,& \none[_{\cdots}] \,& \none[_{\cdots}] & \none[_{\cdots}]  & \none[_{\cdots}]  & \none[_{\cdots}]& \none[_{\cdots}] & & \none[_{\cdots}] & &*(red)&\none[\a_2] \\
\none [^{_{5n}} \,\,\,\,]&   \,& \none[_{\cdots}] \,& \none[_{\cdots}] & \none[_{\cdots}]  & \none[_{\cdots}]&  \none[_{\cdots}]  &\none[_{\cdots}] & \none[_{\cdots}]&\none[_{\cdots}] &\none[_{\cdots}]& &\none[_{\cdots}] & &*(red)&\none[\a_3] \\
\end{ytableau}
$$ 
\vspace{-0.4cm}
\caption{$\D_{6,6n+1}$ et $\D_{6,6n+2}$}\label{D612}
\end{figure}
\vspace{-0.4cm}
We consider the sets:
\begin{align*}
B_1&=\{ \text{paths containing the box } \a_1 \},&\\
B_2&=\{ \text{paths containing the box } \a_2 \text{ and not } \a_1\},&\\
B_3&=\{ \text{paths containing the box } \a_3 \text{ and not }( \a_1 \text{ or } \a_2)\}.&
\end{align*}
We have the following cases : 
\begin{enumerate}[\tb{Case} 1:]
\item
For $B_1$, we find $3 \times b_1$ and $3\times b'_1$ with underlying diagram $\l=(2n,n)$ and $\l'=(2n,n)$, respectively (see eq.\ref{eSPE}). Then,
\[|B_1|=|\D_{3,3n+1}||\D_{3,3n+1}|.\]
\item
For $B_2$, the rectangles $4 \times b_2$ and $2\times b'_2$ such as $\l=(3n,2n,n)$ et $\l'=(n)$ (see eq.\ref{eSPE}).  Then,
\[|B_2|=|\D_{4,4n+1}||\D_{2,2n+1}|.\]
\item 
for $B_3$, the rectangle $5 \times b_3$ such as $\l=(4n,3n,2n,n)$ (see eq.\ref{eSPE}). Then,
\[|B_3|=|\D_{5,5n+1}|.\]
\end{enumerate}
Finally, we obtain:
\begin{align*}
|\D_{6,6n+2}| = \Ct{(6,6n+1)}&+\Ct{(5,5n+1)}  \\
&+\Ct{(4,4n+1)}\Ct{(2,2n+1)}  + \Ct{(3,3n+1)}\Ct{(3,3n+1)}.
\end{align*}
\subsection{Example $\D_{6,9}$.}
For $\D_{6,9}$ the Ferrers diagram is:
\vspace{-0.6cm}
\begin{figure}[h]
$$
\ytableausetup {mathmode, boxsize= 0.8em,centertableaux} 
\ydiagram[*(red)] {1+0,2+1,3+1,4+2,5+2}
* [*(yellow!70) ]{ 1 , 3 , 4 , 6 , 7 }
$$
\vspace{-0.4cm}
\caption{Ferrers diagram $\D_{6,9}$ and $\D_{6}$ }\label{D69}
\end{figure}

\vspace{-0.4cm}
and the diagrams decomposition method. After six iterations the decomposition is:
\begin{align*}
\ytableausetup {mathmode, boxsize= 0.5em,centertableaux} \ydiagram [
*(yellow!70) ]{ 1 , 3 , 4 , 6 , 7 }
&=  
\ydiagram [*(yellow!70) ]{ 1 , 2 , 3 , 4 , 5 } + \ydiagram [ *(yellow!70) ]{ 1 , 2 , 3 , 4 } + \ydiagram [
*(yellow!70) ]{ 1 , 2 , 3 , 4 } + \ydiagram [ *(yellow!70) ]{ 1 , 2 , 3 } \times \ydiagram [
*(yellow!70) ]{ 1 } + \ydiagram [ *(yellow!70) ]{ 1 , 2 , 3 } + \ydiagram [
*(yellow!70) ]{ 1 } \times \ydiagram [ *(yellow!70) ]{ 1 , 2 , 3 } +&\\
&+ \ydiagram [
*(yellow!70) ]{ 1 , 2 } \times\ydiagram [ *(yellow!70) ]{ 1 , 2 } +\ydiagram [ *(yellow!70) ]{ 1 ,
2 } \times \ydiagram [ *(yellow!70) ]{ 1 } + \ydiagram [ *(yellow!70) ]{ 1 , 2 } \times  \ydiagram [*(yellow!70) ]{ 1 } + \ydiagram [ *(yellow!70) ]{ 1 } \times \ydiagram [ *(yellow!70) ]{ 1 ,2 , 3 } + \ydiagram [ *(yellow!70) ]{ 1 } \times  \ydiagram [ *(yellow!70) ]{ 1 , 2 } +\ydiagram [
*(yellow!70) ]{ 1 } \times  \ydiagram [ *(yellow!70) ]{ 1 } \ydiagram [ *(yellow!70) ]{ 1 }
\end{align*}
then,
\begin{align*}
\Hom{\ytableausetup {mathmode, boxsize= 0.3em,centertableaux} \ydiagram [
*(yellow!70) ]{ 1 , 3 , 4 , 6 , 7 }}
& = \Ct{2}^3 + 3\Ct{2}\Ct{3} + \Ct{3}^2 + 3\Ct{2}\Ct{4} + \Ct{4} + 2\Ct{5} + \Ct{6} &\\
&=377 .&
\end{align*}

We can also decompose (see Figure \ref{D6968}), and after four iterations the decomposition is:

\begin{figure}[h]
$$
\ytableausetup {mathmode, boxsize= 0.8em,centertableaux} 
\ydiagram[*(red)] {1+0,2+1,4+0,5+1,6+1}
* [*(yellow!70) ]{ 1 , 3 , 4 , 6 , 7 }
$$
\vspace{-0.4cm}
\caption{Ferrers diagram $\D_{6,9}$ and $\D_{6,8}$ }\label{D6968}
\end{figure}


\begin{align*}
\ytableausetup {mathmode, boxsize= 0.5em,centertableaux} \ydiagram [
*(yellow!70) ]{ 1 , 3 , 4 , 6 , 7 }
&=  
\ydiagram [*(yellow!70) ]{ 1 , 2 , 4 , 5 , 6 } + \ydiagram [ *(yellow!70) ]{ 1 , 2 , 4 , 5 } +  \ydiagram [*(yellow!70) ]{ 1 } \times \ydiagram [ *(yellow!70) ]{ 1 , 2 , 3 }+ \ydiagram [*(yellow!70) ]{ 1 } \times \ydiagram [ *(yellow!70) ]{ 1 , 2  } + \ydiagram [
*(yellow!70) ]{ 1  } \times\ydiagram [ *(yellow!70) ]{ 1 , 3,4 } 
\end{align*}
then,
\begin{align*}
\Hom{\ytableausetup {mathmode, boxsize= 0.3em,centertableaux} \ydiagram [
*(yellow!70) ]{ 1 , 3 , 4 , 6 , 7 }}
& =  \Ct{2}|\D_{4,6}| +\Ct{2}\Ct{3}+\Ct{2}\Ct{4}+  \Ct{5,7} +|\D_{6,8}| &\\
&=2\cdot23+2\cdot5+2\cdot14+66+227 =377.&
\end{align*}
As they are many possible decomposition, finding the shortest one is an open problem.

\subsection*{Final remarks and Acknowledgements} 
The author would like to thank his advisor François Bergeron and Sre\v{c}ko Brlek for their advice and support during the preparation of this paper.
The results presented here are part of J.E. Bla\v{z}ek's Master thesis. Algorithms in SAGE are available at \url{http://thales.math.uqam.ca/~jeblazek/Sage_Combinatory.html}.

\bibliographystyle{splncs03.bst}
\bibliography{Bibliog}

\end{document}